\documentclass[12pt]{amsart}
\usepackage{bm,amsmath,amssymb,eucal}
\usepackage{hyperref}
\input{comment.sty}

\newtheorem{Th}{Theorem}[section]
\newtheorem{Prop}[Th]{Proposition}
\newtheorem{Lem}[Th]{Lemma}
\newtheorem{Cor}[Th]{Corollary}

\theoremstyle{definition}

\theoremstyle{remark}
\newtheorem{Rem}[Th]{Remark}

\def\BN{\mathbb N}

\def\BR{\mathbb R}

\def\SB{\mathcal B}

\def\SE{\mathcal E}

\def\SH{\mathcal H}

\def\SM{\mathcal M}
\def\SN{\mathcal N}

\def\ep{\varepsilon}
\def\ga{\gamma}

\def\si{\sigma}

\def\sminus{\smallsetminus}

\def\({\left(}
\def\){\right)}

\def\oo{\infty}

\begin{document}

\title[Partitions into thin sets]{Partitions into thin sets  and \\ forgotten theorems of Kunugi and Lusin-Novikov}

\dedicatory{}
\author[Grzegorek and Labuda]{Edward Grzegorek and Iwo Labuda}

\address{E. Grzegorek\endgraf
Institute of Mathematics, University of Gda\'nsk\endgraf
Wita Stwosza 57, 80--308 Gda\'nsk,
Poland}
\email{egrzeg@mat.ug.edu.pl}
\address{I. Labuda\endgraf
Department of Mathematics, The University of Mississippi\endgraf
University, MS 38677, USA}

\email{mmlabuda@olemiss.edu}

\date{\today}



\begin{abstract} Let $f$ be a function from a metric space $Y$ to a separable metric space $X$. If $f$ has the Baire property, then it is continuous apart from a 1st category set. In 1935,  Kuratowski asked whether  the separability requirement could be lifted. A full scale attack on the problem took place in the late seventies and early eighties. What was not known then, and what remains virtually unknown today, is the fact that a first impressive attempt to solve the Kuratowski problem, due to Kinjiro Kunugi and based on a theorem of Lusin and Novikov, took place already in 1936. Lusin's  remarkable 1934  Comptes Rendus note soon forgotten, remained unnoticed to this day. We analyze both papers and bring the results to full light.    \end{abstract}

\maketitle

\section{Introduction}

Kazimierz\,Kuratowski and Pavel\,S.\,Alexandrov were friends,  both born in 1896,  and so in 1976  Kuratowski published, in Russian, an inconspicuous  paper  \cite{Kur76} dedicated to  Alexandrov on his 80th birthday. The paper
contained, in particular, the following

\begin{Prop}\label{th:partition} Let $Y$ be a complete separable metric space equipped with a non-trivial finite measure $\mu$. Let
 $\{A_\ga:\ga\in\Gamma\}$ be a partition of $Y$ into 1st category (resp., $\mu$-zero) sets. If  the Continuum Hypothesis  holds, then there exists $\Delta\subset\Gamma$ such that $\bigcup_{\ga\in\Delta}A_\ga$
does  not have the Baire property (resp.,\,is not $\mu$-measurable).
\end{Prop}

Recall that a subset $A$ in a topological space $Y$ is said to be
\textit{(of)\,1st category} (meager) if it is a countable union of nowhere dense sets; $A$ is said to have the \textit{Baire property} if   $A=(O\sminus P)\cup Q$, where $O$ is  open and $P, Q$ are 1st category sets.
Further, a function $f$ from $Y$ into a topological space $X$ has the Baire property (resp., is measurable with respect to a measure $\mu$ on $Y$) if for each open set $O\subset X$ its inverse image $f^{-1}(O)$ has the Baire property (resp., is $\mu$-measurable) in $Y$.
Some authors, in order to stress the analogy (in the sense of, e.g.,\,\cite{Oxt80}) between measure and category, instead of `has the Baire property', write `is $BP$-measurable'.

The result of Kuratowski provoked a flurry of activity. According to an unpublished 1977 manuscript of Prikry (quoted as reference number 27 in \cite{Kou86}),
 its `measure part' for  the Lebesgue measure on $Y=[0,1]$ was already known in ZFC to Solovay in 1970. Further, Prikry reports that in that setting, i.e., in ZFC and for $[0,1]$ with Lebesgue measure, Bukovsk\'y  got  both,  the `measure part' as well as the `Baire part' using models of set theory, and then  follows  with a neat analytic proof (discussed in Sectin 6 below) of a generalization of Bukovsk\'y's result.  Bukovsk\'y's research appeared with a considerable delay, in   \cite{Buk79}.   The same issue of the journal contained another paper, see Emeryk et al.\,\cite{Eme79a}, providing a  topological proof of the ` Baire part' of the statement with $Y$ being  \v Cech complete of weight less than or equal to the continuum.
   For  similar results see also \cite{Brz79}, still in the same issue of the journal as  \cite{Buk79} and \cite{Eme79a}.

The reason for the excitement
was a 1935 problem \cite[Problem 8]{Kur35}, and its measure analogue. It was already known that if $Y$ is a metric space, $X$ is a separable metric space, and  $f:Y\to X$ has the Baire property, then $f$ is continuous apart from a 1st category set. Kuratowski asked whether that conclusion could be extended to a non-separable framework.  Now,  again according to the already quoted manuscript of Prikry, Solovay using his result obtained  a positive answer to the just mentioned measure analogue of the problem in the following form: \textit{if $X$ is a metric space and $f: [0,1]\to X$ is measurable, then $f$ is almost continuous (Lusin
measurable)}. It was rather clear that the `Baire part'  ought to provide some answer to the Kuratowski question as well.

As expected, an answer was delivered by Emeryk et al.\,in  \cite{Eme79b}, a `sibling' paper of \cite{Eme79a} printed next to it,  as follows.
\begin{Th} Let $Y$ be a \v Cech complete
space of weight less than or equal to the continuum. If $f$ is a function having the Baire property from $Y$ into a metric space,  then $f$ is continuous apart from a 1st category set.
\end{Th}

In the realm of  metric spaces the final word
 concerning Kuratowski's problem is due to Frankiewicz and Kunen
 \cite{Fra87}. Of course, as one can guess, research continued (see, e.g., \cite{Fre87} and \cite{Kou86}). But what happened later is irrelevant to our story, so we  stop at this point.

\section{Surprise, surprise}

The section above gave a brief history of initial research provoked by Kuratowski's paper, at least  history which seems to be currently the accepted one. The initial motivation for the present article came out of an unknown\footnote{Actually, the first author learned about the existence of Kunugi's paper in 1983 from \cite{Mor83}. Somehow, the news remained private and did not spread out despite some efforts to the contrary.}  fact that a first attack on  the Kuratowski problem,  due to Kinjiro Kunugi, happened already in 1936 \cite{Kun36}, basing on  \cite{Lus34} and \cite[Suppl\'ement]{Sie34}.

In order to be able to discuss that story, we need some more terminology. A subset $E$ of reals $\BR$ is said to be an
\textit{always (of) 1st category set} if it is 1st category on every perfect set. A subset $E$ of a topological space $Y$ is \textit{2nd category} if it is not 1st category; it is \textit{1st category at (a point)} $x\in Y$ if there is an open set (equivalently, a neighborhood \cite[Sec.\,10,\,V]{Kur66}) $O$ containing $x$, such that $E\cap O$ is 1st category. Further, $E$ \textit{is 2nd category at $x$} if it is not 1st category at $x$, i.e., if $\,E\cap V$ is 2nd category  for each neighborhood $V$ of $x$. If $K$ is a subspace of $Y$ and $E\subset K$, then $D_K(E)$ denotes the set of all points $x$ in $K$ such that $E$ is 2nd category in $x$ relative to $K$. The subscript $K$ is dropped if $K=Y$.  If $E$ is 1st (resp., 2nd) category at each point of a set $P$, we will also say that $E$ \textit{is everywhere 1st} (resp., \textit{2nd}) \textit{category in} $P$.

\medskip

Here is the  `Baire part' of Lusin's statement  in his 1934 CRAS note \cite{Lus34}.

\begin{Th}\label{th:LusinDecomposition}
Let $F$ be a subset of reals which is not always of  1st category. Then
$F$ can be written as the union of   two disjoint subsets $F_1, F_2$ that cannot be separated by  Borel sets.
\end{Th}

Before going ahead with the proof, let us state an observation:
\medskip

{\bf(0)} \textit{Let $A$ and $B$ be subsets of a topological space $Y$, with $A\subset B$. If $A$ is 2nd category (relative to $Y$) at each point of $B$, then  $A$ is so relative to $B$.}

Indeed, suppose to the contrary that there exists an $x\in B$ and an open neighborhood $V$ of $x$ in $B$ such that $V\cap A$ is 1st category relative to $B$ and, therefore, 1st category. Let $V^*$ be an open set such that
$V^*\cap B= V$.  As $V^*\cap A=V\cap A$, we have $x\in V^*\cap B$ and $V^*\cap A$ is 1st category; a contradiction.

\begin{proof}

{\bf (1)} Lusin's starting point is the following reduction. Let $E$  be a set which is not always of 1st category (later, in Part 2 of the proof, $E$ appears first as a subset of reals lying on the axis $OX$ in the plane). Then, without loss of generality, we may assume that $E$ is well-ordered into a transfinite sequence $E=\{x_0,x_1,\dots,x_\ga,\dots\}$ in such a way that every segment of $E$ is an always 1st category set. Moreover, $E$ is contained in a perfect set $P$ such that $E$ is everywhere 2nd category in $P$ relative to $P$ .

Here is the argument. Given $F$ as in the statement of the theorem, let $\{x_0,x_1,\dots,x_\ga,\dots\}$ be a well-ordering  of $F$. If the requirement about its segments is not satisfied, then find the first index $\Gamma$ such that $A=\{x_\ga:\ga<\Gamma\}$ is not always of 1st category.  The latter means that  there exists a perfect set $K$  such that $B= A\cap K$ is 2nd category relative to $K$.
 According to \cite[\S10,VI(14)]{Kur66}, $B=C\cup E$, where $C$ is 1st category relative to $K$ and $E$ is 2nd category relative to $K$ at each of its points. Let $P= D_K(E)$. Then $E\subset P$ and by \cite[\S10,V(11)]{Kur66}, $P$ is the closure of its interior relative to $K$, whence perfect. By the very definition of $P$, the set $E$ is everywhere 2nd category in $P$ relative to $K$ and, by the observation (0) above, relative to $P$. The well-ordering of $E$ is the one induced from $A$, assuring that the assumption about segments is satisfied.

\medskip

{\bf(2)}
By part (1) above, one is faced with a well-ordered set $E=\{x_0,x_1,\dots,x_\ga,\dots\}$  contained in a perfect set $P$; each segment of this transfinite sequence is an always 1st category set, and $E$ is everywhere 2nd category in $P$ relative to $P$.

\textit{This was the setting in which Lusin started the proof. Below is our translation of that proof from the French. We keep to the original closely, including parts that were emphasized by italics.}

\medskip

Locate on the axis $OY$ of the plane $XOY$ a set identical with $E$  and denote by $\SE$, where $\SE\equiv \{\SM(x,y)\}$, the set of points $\SM(x,y)$ of the plane such that $x$ and $y$ belong to $E$ and are such that $Ind(y)<Ind(x)$: here $Ind$ means the index (finite or transfinite) of a point  in  the well-ordered sequence $E$.

Denote by $\SE_{x_0}$ and  $\SE_{y_0}$ two linear sets of all points of $\SE$ located, respectively, on the lines $x=x_0$ and $y=y_0$. It is clear that $\SE_x$ is an always  1st category set for each $x$. Thus, we can write
$$
\SE=\bigcup_{x\in E}\SE_x=\bigcup_{x\in E}\bigcup_{n=1}^\oo\SE^{(n)}_x=\bigcup_{n=1}^\oo\bigcup_{x\in E}\SE^{(n)}_x= \bigcup_{n=1}^\oo H_n,
$$
where $\SE^{(n)}_x$ is nowhere dense in $P$ (parallel to $OY$).  In  contrary, $\SE_y$ is  evidently everywhere  2nd category in $P$ relative to $P$ (parallel to $OX$) when $y\in E$.

It follows that, if
$\pi_1,\pi_2\dots \pi_i\dots$
 is  a sequence formed by subsets of $P$ determined by the intervals with rational end points, for each $y\in E$ there corresponds a pair of positive integers  $(n,i)$ such that $(H_n)_y$ is  everywhere 2nd category in $\pi_i$ (parallel to $OX$). Let
$E_{n,i}$ be  the set of  points  $y\in E$ corresponding to the same  pair $(n,i)$. As there is a countable number of the sets $E_{n,i}$ and as  $E$ (on $OY$) is 2nd category in $P$, one sees easily that there exists a pair $(\nu,\iota)$ such that $E_{\nu,\iota}$ is  dense  in some $\pi_j$ (on $OY)$.

Now,$(H_\nu)_x$ is nowhere dense  in $P$\ (parallel to $OY)$ whatever $x\in E$. Thus, for each $x_0\in E$ there corresponds positive integer  $k$ such that $\pi_k$ in $P$ (on $x=x_0$) contained   in $\pi_j$ does not contain any point of $\pi_k\cap (H_\nu)_{x_0}$.
Denote by  $E^{(k)}$ the set of  points  $x\in E$ that correspond to the same number $k$. As there is countable number of $E^{(k)}$ and as  $E$ is everywhere 2nd category in $P$,  there exists  $\pi_\mu$ in $P$ (on $OX$) contained in $\pi_\iota$  and such that  $E^{(\kappa)}$ \textit{is everywhere 2nd category in} $\pi_\mu$.

But there surely exists a point $y_0\in E_{\nu,\iota}$ that belongs to $\pi_\kappa$ (on $OY$). It is clear that  $(H_\nu)_{y_0}$ is  everywhere 2nd category in  $\pi_\iota$ and, consequently, \textit{is  everywhere 2nd category in} $\pi_\mu$.

 In the end, \textit{we have found two subsets, $E_1$ and  $E_2$  of $E$ that are disjoint and such that each one is everywhere 2nd category in} $\pi_\mu$.
\medskip

This was the end of the original proof. Lusin did not name $E_1$ and $E_2$, but these are $E^{(\kappa)}$ and $(H_\nu)_{y_0}$ that was moved from the line $y=y_0$ onto $OX$.

\medskip

{\bf(3)} By part (1) of the proof, $F=F_0\cup E$ with $F_0\cap E=\emptyset$, and $E\subset P$, where  $P$ is a perfect set  such that  $E$ satisfies the conditions specified in (2). As $E\subset P$ and  Borel sets relative to $P$ are traces of Borel sets, it is sufficient to consider $P$ as the enveloping space. Let $E_1$ and $E_2$ be subsets of $ E$ found in (2).   Suppose there exists   a Borel subset $B$ of $P$ providing separation, i.e.,  let $B\supset E_1$ and  $B$ be disjoint with $E_2$. Clearly, $B$ is 2nd category at each point of $\pi_\mu$ and  its complement  $P\sminus B$, containing $E_2$, is also 2nd category at each point of $\pi_\mu$. Yet, in view of \cite[\S11,\,IV]{Kur66}, $B$ cannot even have the Baire property relative to $P$ and, a fortiori, cannot be a Borel subset of $P$. A contradiction. Consequently, $E_1$ and $E_2$ cannot be separated by Borel sets and the same is true for $F_1=F_0\cup E_1$ and $F_2=F\sminus F_1$.
\end{proof}

\begin{Rem}   Lusin treated as evident not only  part (1), but also  part (3) above; only part (2)  was given in the Note.
 \end{Rem}

Kunugi, who was after Kuratowski's problem, imposed at the beginning of his paper \cite{Kun36}  the following \textit{condition $(\alpha)$} on  a   space $Y$:
\medskip

\textit{Given an ordinal number $\ga$, let   $\{A_\xi\}$ be  a disjoint family of 1st category subsets of $Y$with $\xi$ running through all ordinals less than $\ga$. If
the union $\bigcup_{\xi<\ga} A_\xi$ is 2nd category, then one can break the union into two disjoint parts $\bigcup_{\xi'} A_{\xi'}$ and $\,\bigcup_{\xi''} A_{\xi''}$   in such a way that each part is everywhere 2nd category in a common open subset of $\,Y$.}
\medskip

For a proof that $(\alpha)$ holds in a separable metric space, he referred to Lusin \cite{Lus34} and to Sierpi\'nski
 \cite[Suppl\'ement]{Sie34}.

Now Sierpi\'nski, referring to a private communication  from Lusin, stated in his Suppl\'ement  the following `Lemme'.
\begin{Lem}\label{prop:SierpLemma} Let $E$ be a 2nd category subset of an interval $I$. There exist an interval $J\subset I$ and two disjoint subsets $E_1$ and $ E_2$ of $E$ that are  2nd category at each point of~$J$.
\end{Lem}

Lemma \ref{prop:SierpLemma}  ought to be compared with the part (2) of the proof of Theorem~\ref{th:LusinDecomposition}. Sierpi\'nski's assumption is stronger, and the conclusion too, because Lusin's conclusion is `at each point of $J\cap P$' for some perfect set $P$, and Sierpi\'nski has `at each point of $J$'.

\textit{The condition $(\alpha)$ of Kunugi is nothing else, but
Lemma~\ref{prop:SierpLemma} (in the realm of separable metric spaces), in which points of the set $E$ are replaced by disjoint subsets of 1st category.}

Whether Kunugi thought that the passage from `points' to  `disjoint 1st category subsets' posed no problem, or knew how to accomplish it and yet did not bother to give a proof, we will never know. What is not in doubt, is the fact that it was the `Lemme' in \cite{Sie34}, which prompted him to formulate the condition $(\alpha)$.

  We mention this, because we first  considered Sierpi\'nski's proof, which is rather long and does not bear any resemblance with the arguments in (2) above. It was not visible how to generalize it, to get $(\alpha)$. Only then, we focused on  Lusin's Note,  despite the fact that a direct translation of the result there would not produce the condition $(\alpha)$.

It turned out that the technique developed in (2) permits it to obtain Lemma~\ref{prop:SierpLemma} as well. Indeed, let us replace the assumption `$F$ is not always 1st category' in Theorem~\ref{th:LusinDecomposition},  by `$F$ is 2nd category'.  The reduction described in part (1) of its proof can now be performed without the intermediate step involving $K$. Hence, one ends up with
$E\subset P$ and
 $P=D(E)$. It follows that  the set $\pi_\mu$ appearing at the end of Lusin's proof is a non-empty set of the form $(a,b)\cap D(E)$. It must contain an interval $J$, since $D(E)$ is the closure of its non-empty interior
(\cite[\S10,V(11)]{Kur66}),
  which gives Sierpi\'nski's conclusion.

\bigskip

 \section{Disjoint families of 1st category sets}

An examination of the technique used  in the proof of Theorem~~\ref{th:LusinDecomposition} allows us to prove  the following result, which confirms that Kunugi's claim about the validity of the condition $(\alpha)$  was correct. The  relevant topological fact needed for the proof is  the existence of a countable base for open sets, i.e., that the space is \textit{2nd countable}.

\begin{Th}\label{prop:Lusinlemma2} Let $\{A_\ga:\ga\in\Gamma\}$ be a family of disjoint 1st category subsets of a 2nd countable topological space $Y$. If  the union $E=\bigcup_{\ga\in\Gamma}A_\ga$
is  2nd category, then there exist $\Delta\subset\Gamma$ and an open set $O$  such that $\bigcup_{\ga\in\Delta}A_\ga$ and $\bigcup_{\ga\in\Gamma\sminus\Delta}A_\ga$ are everywhere 2nd category in $O$. In consequence, $\bigcup_{\ga\in\Delta}A_\ga$ and $\bigcup_{\ga\in\Gamma\sminus\Delta}A_\ga$ cannot be separated by  sets having the Baire property.
\end{Th}

\begin{proof}  We may assume without loss of generality (cf.\,\,Part 1 in the proof of Theorem~\ref{th:LusinDecomposition})  that $\Gamma$ is an ordinal, $\ga<\Gamma$, and that for each $\ga<\Gamma$, the set $\bigcup_{\beta<\ga}A_\beta$ is of 1st category.

Consider the product $ Y\times Y$. It will be convenient to speak in terms of the  $XOY$ coordinate system with the first copy of $Y$ as  the `horizontal x-axis', and the second  copy of $Y$ as the `vertical  $y$-axis'.  For any $\beta<\Gamma$, set $B_\beta  =\bigcup_{\ga<\beta}A_\ga$. Define $\SE\subset Y\times Y$ by
$$
\SE= \bigcup_{\ga<\Gamma}(A_\ga\times B_\ga).
$$
Note that the `vertical' set $\SE_{x_0}$ lying on the `line $x=x_0$' (i.e., on
$\{(y,x_0):y\in Y\})$ is 1st category for every $x_0\in Y$. Write
$$
\SE=\bigcup_{x\in E}\SE_x=\bigcup_{\ga<\Gamma}\bigcup_{x\in A_\ga}(\{x\}\times B_\ga)=\bigcup_{\ga<\Gamma}\bigcup_{n=1}^\oo (A_\ga\times B^{(n)}_\ga)=\bigcup_{n=1}^\oo H_n,
$$
where $H_n=\bigcup_{\ga<\Gamma}(A_\ga\times B^{(n)}_\ga)$ and the sets $B^{(n)}_\ga$ are nowhere dense in $Y$ (`parallel' to $y$-axis).

Let $$O_1,O_2\dots O_i\dots$$ be a countable base of open sets of $Y$.
 For each $y_0\in E$, one can find   $\alpha$  such that $y_0\in A_\alpha$. Then,  $y\in B_\beta$ for any $\beta>\alpha$  and
$\SE_ {y_0}=\bigcup_{\alpha<\beta<\Gamma}A_\beta$. In particular,
the `horizontal' set $\SE_y$ lying on the ` line $y=y_0$' (i.e., on
 $\{(x,y_0): x\in Y\}$) is 2nd category for each $y_0\in E$.

As $\SE_y=(\bigcup_{n=1}^\oo H_n)_y$ is 2nd category, there exists $n\in\BN$ such that $(H_n)_y$ is 2nd category and so, by \cite[\S10 (7) and (11)]{Kur66}, $Int D((H_n)_y)$ is non-empty. It follows that for every $y\in E$,
 one can find a pair $(n,i)$ of naturals such that $(H_n)_y$ is everywhere 2nd category in $O_i$ (`parallel' to $x$-axis).
Let
$$
E_{n,i}=\{y\in E: (H_n)_y\text{ is everywhere 2nd category in  } O_i  \},\leqno(*)
$$
  i.e., the set of  all points  $y\in E$ corresponding to the same pair  $(n,i)$.

As $E=\bigcup\{E_{n,i}: n\in \BN, i\in \BN\}$ and  $E$ (on $y$-axis) is 2nd category, there exists a pair $(\nu,\iota)$ such that $E_{\nu,\iota}$ is  dense  in some $O_j$ (on $y$-axis).

Further, as $(H_\nu)_x$ is nowhere dense in $Y$ (`parallel' to $y$-axis) for any $x\in E$, it follows that  for each $x_0\in E$ there exist a natural number $k$ and $O_k$ (on the line $x=x_0$) contained   in $O_j$  (moved from $y$-axis onto the line $x=x_0$) such that  $O_k\cap (H_\nu)_{x_0}=\emptyset$.

Denote by  $E^{(k)}$ the set of points $x\in E$ that correspond to $k\in\BN$.  If $y\in E_{\nu,\iota}$, in view of $(*)$, $(H_\nu)_y$ is everywhere 2nd category in $O_\iota$. A fortiori $(E)_y$ is so there too. But $E=\bigcup_{k=1}^\oo E^{(k)}$, so there must exist $O_\mu$ (on $x$-axis) contained in $O_\iota$ and $\kappa$ such that $E^{(\kappa)}$is everywhere 2nd category in $O_\mu$.

Now suppose $x_1 \in E^{(\kappa)}$ and $x_2$ is such that $x_1,x_2$ belong to the same $A_\ga$. Then $O_\kappa\cap H_\nu=\emptyset$ on $x=x_1$ implies the same on $x=x_2$ and therefore $x_2\in E^{(\kappa)}$, i.e., $E^{(\kappa)}$ is \textit{saturated} (in the sense that it defines $\Delta\subset\Gamma$ such that $E^{(\kappa)}=\bigcup_{\ga\in\Delta}{A_\ga}$).

But, as $E_{\nu,\iota}$ was dense in $O_j$, there surely exists a point $y_0\in E_{\nu,\iota}$ which belongs to $O_\kappa$ (on $y$-axis). Since  $(H_\nu)_{y_0}$ is  everywhere 2nd category in  $O_\iota$, it  is so in $O_\mu$ as well.

Again, suppose that $(x_1,y_0)\in(H_\nu)_{y_0}$ and  $x_2$ belongs to the same $A_\ga$ as $x_1$. Then, from the definition of the sequence $(H_n)$, also $(x_2,y_0)\in (H_\nu)_{y_0}$. Hence $(H_\nu)_{y_0}$ is saturated and therefore defines $\Delta'\subset\Gamma$ such that $\bigcup_{\ga\in\Delta'}A_\ga$ (moved on the line $y=y_0$) equals $(H_\nu)_{y_0}$. By construction, $E^{(\kappa)}$ and $(H_\nu)_{y_0}$ (moved from the line $y=y_0$ onto $x$-axis)  are disjoint and, consequently, also $\Delta$ and $\Delta'$ are disjoint. Moreover,  the corresponding unions are everywhere 2nd category in $O_\mu$.
It is also clear that $\bigcup_{\ga\in\Gamma\sminus\Delta} A_\ga$, containing $\bigcup_{\ga\in\Delta'}A_\ga$, must be everywhere 2nd category in $O_\mu$. Set $O_\mu=O$.

Finally, observe that there is no  $G\subset Y$ having the Baire property such that
$G\supset\bigcup_{\ga\in\Delta} A_\ga$ and
$G\cap\bigcup_{\ga\in\Gamma\sminus\Delta}A_\ga=\emptyset$. Indeed, suppose the contrary. Then $G$, containing $\bigcup_{\ga\in\Delta}A_\ga$, must be everywhere 2nd category in $O$ and its complement $G'$, containing $\bigcup_{\ga\in\Gamma\sminus\Delta}A_\ga$, is also everywhere 2nd category in $O$. Hence, in view of  \cite[\S11,\,IV]{Kur66}, $G$ cannot have Baire property. A contradiction. In particular,
$\bigcup_{\ga\in\Delta}A_\ga$ does not have the  Baire property.
\end{proof}

\begin{Cor} Let $E$ be 2nd category subset of a 2nd countable $T_1$-space $Y$ without isolated points. Then $E$ can be written as the union of two disjoint subsets that cannot be separated by  sets having the Baire property.
\end{Cor}

\begin{proof} A topological space is  $T_1$ if its  one-point subsets are closed. Thus, if a singleton  is not an isolated point, then it is nowhere dense, and we can apply the above Theorem for a partition of $E$ into points.
\end{proof}

According to \cite[\S40,\,II]{Kur66}, a separable metric space $Y$ is said to be an \textit{always of  1st category} space if every dense in itself subset of $Y$ is 1st category in itself. A subset in a \textit{Polish space} (i.e., a 2nd countable topological space that admits a complete metric) $Y$  is always of  1st category if\-f it is 1st category on every perfect subset of $Y$ \cite[\S40,\,II, Theorem 1]{Kur66}. Consider also the $\si$-algebra of sets having the Baire property with the restricted
sense (\cite[\S11,\,VI]{Kur66}, that is,  sets whose traces on every subset of $Y$ have the Baire property relative to that subset.
 The next result is a generalization of Theorem~\ref{th:LusinDecomposition} to partitions.

\begin{Th}\label{th:BR} Let $\{A_\ga:\ga\in\Gamma\}$ be a family of disjoint always 1st category subsets of a Polish space $Y$. If  the union $E=\bigcup_{\ga\in\Gamma}A_\ga$
is  not an always 1st category  subset of $Y$, then there exists $\Delta\subset\Gamma$ such that  $\bigcup_{\ga\in\Delta}A_\ga$ and $\bigcup_{\ga\in\Gamma\sminus\Delta}A_\ga$ cannot be separated by sets having the Baire property in the restricted sense.
\end{Th}

\begin{proof} If $E$ is not always of 1st category, then there exists a perfect subset $K$ of $Y$ such that $E\cap K$ is 2nd category relative to $K$. Consider the family $\{A_\ga \cap K:\ga\in\Gamma\}$, and note that it is a family of 1st category subsets of $K$. By applying Theorem~\ref{prop:Lusinlemma2}, we will find $\Delta\subset\Gamma$ such that  $\bigcup_{\ga\in\Delta}A_\ga\cap K$ and $\bigcup_{\ga\in\Gamma\sminus\Delta}A_\ga\cap K$ cannot be separated by   sets having the Baire property relative to $K$. Now, suppose there exists a set $C$  containing $\bigcup_{\ga\in\Delta} A_\ga$, which is disjoint with the union of $A_\ga$'s over $\Gamma\sminus\Delta$ and, moreover, has  the Baire propewrty in the resticted sense. Then $C\cap K$ contains
$\bigcup_{\ga\in\Delta}A_\ga\cap K$, is disjoint with $\bigcup_{\ga\in\Gamma\sminus\Delta}A_\ga\cap K$ and, therefore, separates these two unions on $K$. This contradicts   Theorem~\ref{prop:Lusinlemma2},   because  $C\cap K$ has the Baire property relative to $K$. \end{proof}

\begin{Cor}\label{cor:last} Let $Y$ be a Polish space without isolated points and $E\subset Y$. If $E$ is not an always 1st category set, then it can be written as the union of two disjoint sets that cannot be separated by sets having the Baire property in the restricted sense.
\end{Cor}

\section{Kunugi's Theorem}

\begin{Th}Let $X$ be a topological space satisfying the condition $(\alpha)$, $Y$ a metric space, and $f:X\to Y$ a function having the Baire property. Then $f$ is continuous apart from a 1st category set.
\end{Th}

Below is Kunugi's proof. We keep to the original and its notation closely. Although a few references to \cite[\S\S10 and 11]{Kur66} would be helpful (at that time it would be \cite{Kur33}, to which in fact Kunugi referred while giving the definition of the Baire property), the proof is totally `modern', correct, and we think -- quite impressive, taking into account the date of its discovery.

\begin{proof}
\textit{Kunugi calls a set to be of type $G_\rho$ if it is a set-theoretical difference of two open sets.  He begun his proof by stating}:

As $Y$ is a metric space, by a result of Montgomery \cite{Mon35},
it is possible to find a sequence $Y^n\, (n=1,2,3,\dots)$ such that $Y=\bigcup_{n=1}^\oo Y^n$ and each $Y^n$  decomposes into a transfinite sequence $Y^n_\xi$ of $G_\rho$-sets with distance $\rho(Y^n_\xi, Y^n_{\xi'})>\frac1n$ for $\xi\neq\xi'$. Besides, we can assume that given a base of neighborhoods  for $Y$, $Y^n_\xi$ is contained in some of those neighbourhoods whatever $n$ and $\xi$.

\textit{The above statement is correct. However, its derivation from \cite{Mon35} being not quite immediate, we provide a proof in the Appendix (Section 7 below). Kunugi continues}:

Letting $(\ep_k)$ to be a sequence of positive reals converging to $0$, we assume that the diameter of $Y^n_\xi$ is less than $\ep_k$.  Set $Y=\bigcup_{n=1}^\oo Y^n(\ep_k)$ and $Y^n(\ep_k)=\bigcup_\xi Y^n_\xi(\ep_k)$.

By our assumption, $f^{-1}(Y^n_\xi(\ep_k))$ has the Baire property. We can thus write
$$f^{-1}(Y^n_\xi(\ep_k))=(G^n_\xi(k)\sminus P^n_\xi(k))\cup Q^n_\xi(k),
$$
where $G^n_\xi(k)$ is  open,  while $P^n_\xi(k)$ and $ Q^n_\xi(k)$ are 1st category sets.

Write $D(n,k)=\bigcup_\xi (P^n_\xi(k)\cup Q^n_\xi(k))$, and assume that we can find $P^n_\xi(k)$ and $ Q^n_\xi(k)$ in such a way that $D(n,k)$ is 1st category for all $n$ and $k$. We are going to show that $f$ is continuous when one neglects the set $D=\bigcup_n\bigcup_k D(n,k)$. Indeed, let  $G$ be an open subset of $Y$. Consider
$\bigcup'Y^n_\xi(\ep_k)$, where the union $\bigcup'$  is taken over all $n,k$ such that $Y^n_\xi(\ep_k)\subset  G$. For a point $p\in G$ find a ball of radius $\ep_k$ contained in $G$. Then, there exist $n$ and $\xi$ such that $p\in Y^n_\xi(\ep_k)\subset G$. Hence $G=\bigcup' Y^n_\xi(\ep_k)$. It follows that
$
f^{-1}(G)\cap (X\sminus D)=\bigcup' G^n_{\xi(k)}\cap(X\sminus D).
$
Hence $f^{-1}(G)$ is open in $X\sminus D$.

We have shown that $f$ is continuous  apart from  1st category set $D$. It remains to prove that  it is possible to choose  the sets $P^n_\xi(k)$ and $Q^n_\xi(k)$ in such a way  that $D(n,k)$ be of 1st category.

Let us now say that a transfinite sequence of disjoint sets having the Baire property

$$
X_0,X_1, X_2\dots X_\xi\dots   \leqno{(*)}
$$
\textit{satisfies the condition} $(P)$,  if we can write $X_\xi=(G_\xi\sminus P_\xi)\cup Q_\xi$, with 1st category $P_\xi$ and $Q_\xi$, so that the union $\bigcup_\xi (P_\xi\cup Q_\xi)$ be  of 1st category. Moreover, the sequence $(*)$ \textit{satisfies the condition $(P)$ at a point $p$}, if there exists  a neighborhood $V(p)$ such that  the sequence
$$
X_0\cap V(p), X_1\cap V(p),\dots X_\xi\cap V(p)\dots
$$
satisfies the condition $(P)$. The set of all points of $X$ at which the sequence $(*)$ satisfies $(P)$  will be denoted $\{X_\xi\}_I$, and
$X\sminus \{X_\xi\}_I$ will be denoted $\{X_\xi\}_{II}$.
Now, one can see without much pain that

\begin{enumerate}
\item[\rm{(I)}]if $\{X_\xi\}_{II}=\emptyset$, the sequence $(*)$ satisfies the condition $(P)$;
\item[\rm{(II)}] if $\{X_\xi\}_{II}\neq\emptyset$, $\{X_\xi\}_{II}$ is 2nd category at each point of $\{X_\xi\}_{II}$;
\item[\rm{(III)}] $\{X_\xi\}_{II}$ is the closure of an open set.
\end{enumerate}

Suppose now $\{X_\xi\}_{II}\neq\emptyset$. By (III), there exists an open set $G$ such that $\{X_\xi\}_{II}=\overline G$. We claim that $G\cap X_\xi$ is 1st category for all $\xi$. Indeed,  otherwise $G\cap X_{\xi_0}$ is 2nd category for some $\xi_0$. Consequently, there exist $p\in G$ and a neighborhood $V(p)\subset G$ of $p$ such that $G\cap X_{\xi_0}$ is 2nd category at each point of $V(p)$.  $X_{\xi_0}$ having the Baire property, $V(p)\sminus X_{\xi_0}$ is 1st category. As $\bigcup_{\xi\neq\xi_0} X_\xi$ is
disjoint with $X_{\xi_0}$, one has $\bigcup_{\xi\neq\xi_0} V(p)\cap X_\xi\subset  V(p)\sminus X_{\xi_0}$, and therefore is 1st category. Defining $G_{\xi_0}= V(p), P_{\xi_0}=V(p)\sminus X_{\xi_0},\, Q_{\xi_0}=\emptyset$, and $G_\xi=P_\xi=\emptyset$,  $Q_\xi=V(p)\cap X_\xi$  for $\xi\neq\xi_0$,  we see that the sequence $(X_\xi\cap V(p))$ satisfies the condition $(P)$ at $p$, in contradiction with the choice of the point $p$.

It follows that if  $\{X_\xi\}_{II}\neq\emptyset$, then there  is a non-empty open set $G$ such that $(i)\ G\cap X_\xi$ is 1st category for each $\xi$; $(ii)$ $\ G\cap X_\xi$ are disjoint; $(iii)$ $ G\cap\bigcup_\xi X_\xi$ is everywhere 2nd category in $G$.
If so, our assumption allows us to break the sequence $(*)$ into two disjoint pieces: $\bigcup_{\xi'} X_{\xi'}\cap G$ and $\bigcup_{\xi''} X_{\xi''}\cap G$ in such a way that they are everywhere 2nd category in an open set.
Thus, the set  $\bigcup_{\xi'} X_{\xi'}$ does not have Baire property.

If we define $X_\xi=f^{-1}(Y^n_\xi(\ep_k))$, and if $\{X_\xi\}_{II}\neq\emptyset$, as the set  $\bigcup_{\xi'}Y^n_{\xi'}(\ep_k)$ is of type $G_\rho$, the function $f$ would  not have the Baire property, in contradiction with our hypothesis.

We see, therefore, that $\{f^{-1}(Y^n_\xi)(\ep_k)\}_{II}=\emptyset$. This means, by (I), that the transfinite sequence  $( f^{-1}(Y^n_\xi)(\ep_k))_\xi$ satisfies the condition $(P)$,  whatever $n$ and $k$ ($n,                            k=1,2,3,\dots)$.
\end{proof}

\section{Disjoint families of measure zero sets}

The second part of \cite{Lus34} deals with a measure analogue of the first part. Instead of 2nd category subsets of reals, the sets of positive outer  Lebesgue measure $m_e$ appear.  Lusin's initial reduction corresponds to  the step (1) in the proof of Theorem~\ref{th:LusinDecomposition}. He ends up with a well-ordered subset  $E=\{x_0,x_1,\dots,x_\ga,\dots\}$  of a perfect set $P$ such that $m_e(E)=m(P)<\oo$ and  every segment of $E$ is of measure zero. The definition of the `matrix' $\SE\equiv \{\SM(x,y)\}$ remained the same and   the proof, which we will now adapt for partitions, followed.

 Again, the existence of a countable base of open sets  was essential. Consequently, as in Section 3,  $Y$ \textit{is assumed to be 2nd countable}. Our measure $m$ is \textit{the completion of a finite regular Borel measure} on $Y$ and its \textit{outer measure} is denoted by $m_e$.

The passage from `points' to `disjoint measure zero sets' can  be done without much pain. As in the `Baire part', sets $B_\ga=\bigcup_{\alpha\leq\ga}A_\alpha$  for  $\ga<\Gamma$ need to be defined and, instead of  the  matrix $\SM(x,y)$, one has to consider its analogue in which points $x_\ga$ are replaced by sets $A_\ga$. Once this done, the proof below is a modification of    the  proof in \cite{Lus34}.

We add some auxiliary results in order to make the presentation  reasonably accessible.
 For the upcoming lemma see, for instance, \cite[\S\,12]{Hal50}.

\begin{Lem}\label{lem:meshull}
Let $A\subset B$ be subsets of $\-Y$, the set $B$ being measurable. The following conditions are equivalent.
\begin{enumerate}
\item[\rm{(a)}] $m_e(A)=m(B)$.
\item[\rm{(b)}] If $G$ is measurable and $G\subset B\sminus A$, then $m(G)=0$.
\end{enumerate}
\end{Lem}

A set $B$ described by the above lemma is called an ($m$-measurable)\textit{ envelope} of $A$; it is often  denoted by $\tilde A$.
Note that $m(\tilde A)=\inf\{m(C): A\subset C \}$, where $C$ run over measurable sets.

The following lemma (\cite[\S\,12(4)]{Hal50}) can be obtained as an easy  application of the notion  of  envelope.
\begin{Lem}\label{lem:sci}
 If $A_1\subset A_2\subset\dots $ is a sequence of subsets of $Y$, then $m_e(A_n)\uparrow m_e(\bigcup_{n=1}^\infty A_n)$.
\end{Lem}

The next lemma is \cite[Proposition 6.1.322]{HR48} or \cite[Proposition 11.2.5]{Mun53}.

\begin{Lem}\label{lem:sep1}
Let $(A_n)$ be a sequence of subsets of $Y$. If there exist  disjoint measurable sets $B_n$  such that  $A_n\subset B_n$ for each $n\in\BN$, then  $m_e(\bigcup_{n=1}^\oo A_n)=\sum_{n=1}^\oo m_e(A_n)$.
\end{Lem}
\begin{proof}
One can assume that  $B_n=\tilde A_n$, an  envelope of $A_n$, for $n\in\BN$. We have $m_e(\bigcup_n A_n)= m(\bigcup \tilde A_n)$. Indeed, let $G$ be measurable and $G\subset (\bigcup_n\tilde A_n\sminus \bigcup_n A_n)$. Then $G\cap \tilde A_n$ is a measurable set contained in $\tilde A_n\sminus A_n$ and therefore, by Lemma~\ref{lem:meshull}, it is of measure zero. Hence $G=\bigcup_n G\cap\tilde A_n$ is also of measure zero. Applying Lemma~\ref{lem:meshull} again, $m_e(\bigcup_n A_n)=m(\bigcup_n\tilde A_n)=\sum_n m(\tilde A_n)=\sum_n m_e(A_n)$.
\end{proof}

\begin{Cor}\label{cor:sep2} Suppose $(A_n)$ is a sequence of subsets of $Y\,$that can be pairwise separated by measurable sets. Then $m_e(\bigcup_n A_n)=\sum_n m_e(A_n)$.
\end{Cor}

\begin{proof}
By  assumption, for each $n\neq m$, there is a measurable set $B_{nm}$ such that one has $A_n\subset B_{nm}$ and $B_{nm}\cap A_m=\emptyset$. For  $n\in\BN$, write  $C_n=\bigcap\{B_{nm}: m\in\BN, m\neq n\}$. Then, for each $n$,
 we have $$
 A_n\subset C_n \text{ and } C_n\cap A_m=\emptyset\, \text{ for  } m\neq n, \leqno{(+)}
  $$
 and $C_n$ are measurable. Define $B_n=C_n\sminus \bigcup_{k=1} ^{n-1}C_k$. It follows from $(+)$ that $B_n$'s satisfy the assumption of
Lemma~\ref{lem:sep1}.
\end{proof}

Denoting by $m\times m$ the product measure in $Y\times Y$,  we have
\begin{Lem}\label{lem:Fubini}
Let  $Z\subset Y\times Y$, and $a, b$ be positive reals.
 If $m_e(\{y:\, m_e(Z_y)>b\})>a$, then $(m\times m)_e(Z)>ab$.
\end{Lem}

\begin{proof}

Let $B$ be a measurable (with respect to the product $\si$-algebra) set containing $Z$ such that $(m\times m)(B)=(m\times m)_e(Z)$. We therefore have
$m(\{y:\, m(B_y)>b\})>a$. By the Fubini theorem,
$$
(m\times m)(B)=\int_{Y\times Y}1_B\,d(m\times m)(x,y)=\int_Y\left(\int_Y 1_B (x,y)\,dm(x)\right)dm(y)=\int_Y f(y)\,dm(y),
$$
\newline
where $m(\{y: f(y)>b\}>a$. So $\int_Y f(y)\,dm(y)>ab$ and
 $(m\times m)_e(Z)=(m\times m)(B)>ab$.

\end{proof}

\begin{Lem}\label{lem:sublemma} Let $Z_1$ and $Z_2$ be subsets of $Y$. Then
$
(m\times m)_e(Z_1\times Z_2)_e= m_e(Z_1)m_e(Z_2).$

\end{Lem}

\begin{proof}If $C\subset Y\times Y$ is   measurable containing  $Z_1\times Z_2$ and $(m\times m)(C)= m_e(Z_1\times Z_2)$, then
 $m_e(Z_1)m_e(Z_2)=m(\tilde Z_1)m(\tilde Z_2)=(m\times m)(\tilde Z_1\times\tilde Z_2)\geq (m\times m)_e(Z_1\times Z_2)=(m\times m)(C)$. But $(m\times m)(C)\geq m_e(Z_1)m_e(Z_2)$ by Lemma~\ref{lem:Fubini},  and the  asserted equality follows.
\end{proof}

\begin{Th}\label{th:lusinmes}
Let $\{A_\ga:\ga\in \Gamma\}$ be a disjoint family of measure zero subsets of  a 2nd countable topological space $Y$ and   $E=\bigcup_{\ga\in\Gamma}A_\ga$. If $m_e(E)>0$, then there exists $\Delta\subset \Gamma$ such that
$\,\bigcup_{\ga\in\Delta}A_\ga$ and $\,\bigcup_{\ga\in\Gamma\sminus\Delta}A_\ga$ cannot be separated by measurable sets.
\end{Th}

\begin{proof}
Similarly as before, we may assume without loss of generality that $\Gamma$ is an ordinal, $\ga<\Gamma$, and that for each $\ga<\Gamma$, one has  $m(\bigcup_{\beta<\ga}A_\beta)=0$.

Let $\{O_n:n\in\BN\}$ be a base of open sets in $Y$ and, for each $\ga<\Gamma$, let $B_\ga=\bigcup_{\alpha\leq\ga}A_\alpha$.
 As $m(B_\ga)=0$, we can find basic open sets  $O_n^\ga$ such that $B_\ga\subset\bigcup_{n=1}^\oo O_n^\ga$ and $m(\bigcup_{n=1}^\infty O_n^\ga)<\ep$, where $\ep>0$ and is as small as we wish.
Using the conventions of the proof of Theorem~\ref{prop:Lusinlemma2}, write
 $$
\SE=\bigcup_{x\in E}\SE_x=\bigcup_{\ga<\Gamma}\bigcup_{x\in A_\ga}(\{x\}\times B_\ga)=\bigcup_{\ga<\Gamma}\bigcup_{n=1}^\oo (A_\ga\times B^{(n)}_\ga)=\bigcup_{n=1}^\oo H_n,
$$
where $H_n=\bigcup_{\ga<\Gamma}(A_\ga\times B^{(n)}_\ga)$ and  $B^{(n)}_\ga=O_n^\ga\cap B_\ga$.

For each $y\in E$ and $\ep>0$, there exists $n\in\BN$ such that
 $$
 m_e \left((\bigcup_{i=1}^n H_i)_y\right)>m_e(E)-\ep .
 $$
Indeed, $(\bigcup_{i=1}^n H_i)_y\uparrow(\bigcup_{i=1}^\oo H_{i})_y=\SE_y$, whence
by Lemma~\ref{lem:sci},

 $$
 m_e\left((\bigcup_{i=1}^n H_i)_y\right)\uparrow m_e\left((\bigcup_{i=1}^\oo H_i)_y\right)=m_e(\SE_y)=m_e(E).
$$
 Consequently, by applying Lemma~\ref{lem:sci} again,  there exists $N\in\BN$ such that
$$
m_e\left(\{y:m_e((\bigcup_{i=1}^N H_{i})_y)>m_e(E)-\ep\}\right)>m_e(E)-\ep. \leqno{(*)}
$$
To see this, let $A^N= \{y:m_e((\bigcup_{i=1}^N H_{i})_y)>m_e(E)-\ep\}$. Clearly,  $\bigcup_{N=1}^\oo A^N=E$, and so $m_e(A^N)\uparrow m_e(E)$ with $N\to\oo$.

For $N$ just found, denote $\SH=\bigcup_{i=1}^N H_i$. From $(*)$ and
  Lemma~\ref{lem:Fubini}, we have

$$
(m\times m)_e(\SH)> (m_e(E)-\ep)^2. \leqno{(**)}
 $$
 By the definitions of  $\SH$ and   $(O_n^\ga)_{n=1}^\oo\,$,  for each $x\in A_\ga$ with $\ga<\Gamma$ one gets
$$
\SH_x\subset O_1^\ga \cup O_2^\ga\cup \dots O_N^\ga.
$$
For each sequence    $(U_1,U_2,\dots U_N)$ of basic open sets, define
$$
E^{U_1 U_2\dots U_N}=\bigcup\{A_\ga:(O_1^\ga,O_2^\ga, \dots O_N^\ga) = (U_1,U_2,\dots U_N)\}.
$$
These are disjoint sets. Ordering those that are non-empty into a sequence,
one gets$$
E^1,E^2,\dots E^k\dots
$$
Now, $E=\bigcup_k E^k$ and we claim that the sets  $E^k$ cannot be separated  by measurable sets.
Indeed, if the separation were possible,  by Corollary~\ref{cor:sep2}  one would have the equality
$$
m_e(E)=\sum_{k=1}^\oo m_e( E^k).
$$
But
$$
\SH\subset\bigcup_{k=1}^\oo (E^k\times \{U_1^k\cup U_2^k\cup\dots U_N^k\}),
$$
where $E^k=E^{U_1^k U_2^k\dots U_N^k}$.
Moreover, one has

$$
\begin{aligned}
(m\times m)_e(\SH)\leq \sum_k (m\times m)_e(E^k\times\{U_1^k\cup U_2^k\cup\dots U_N^k\})\\
=\sum_k m_e(E^k)m(\{U_1^k\cup\- U_2^k\cup\dots U_N^k\})
<\ep\,\sum_k m_e(E^k),
\end{aligned}
$$
where the equality in the middle holds by Lemma~\ref{lem:sublemma}.
Hence
$$
(m\times m)_e(\SH)\leq \ep\cdot m_e(E),
$$
which gives a contradiction with $(**)$, in view of the arbitrariness of  $\ep$.

There must exist, therefore, indices $i$ and $l$ such that  $E^i$ and $E^l$ cannot be separated by measurable sets. A fortiori, $E^i$ and $\bigcup_{k\neq i} E^k$ cannot be so separated.    Let $\Delta$ be the set of $\ga$'s  determined by $A_\ga$'s corresponding to $E^i$.
\end{proof}
In particular, if $m(\bigcup_{\ga\in\Gamma} A_\ga)>0$, then there exists $\Delta\subset\Gamma$
such that $\bigcup_{\ga\in\Delta} A_\ga$ is not measurable.

\begin{Cor}\label{cor:LNT} Let $E$ be a subset of a 2nd countable $T_1$-space $Y$ equipped with the completion of a finite regular Borel measure $m$  vanishing on points. If $m_e(E)>0$, then $E$ can be written as the union of two disjoint subsets that cannot be separated by measurable sets.
\end{Cor}
\begin{proof} One point subsets of $Y$, being closed, are in the domain of the measure $m$ and so of measure zero. We may apply the above Theorem to the partition of $E$ into points.
\end{proof}

A subset $E$ of $Y$ is said to be \textit{universally measurable} with respect to a non-empty family  of  measures on $Y$ if it is measurable with respect to each measure $\mu$ in the family; it is  \textit{universally null} with respect to this family if $\mu(E)=0$ for each measure $\mu$ in the family.  Let $\SM$ be a  family of the completions of  regular Borel measures on  $Y$. One shows easily the following measure analogue of Theorem 3.3.

\begin{Th}\label{th:universal} Let $\{A_\ga:\ga\in\Gamma\}$ be a family of disjoint universally $\SM$-null subsets of  a 2nd countable topological space. If   $E=\bigcup_{\ga\in\Gamma}A_\ga$
is  not universally $\SM$-null, then there exists $\Delta\subset\Gamma$ such that  $\bigcup_{\ga\in\Delta}A_\ga$ and $\bigcup_{\ga\in\Gamma\sminus\Delta}A_\ga$ cannot be separated by universally $\SM$-measurable sets.
\end{Th}


Denote by $\SN$ the  family of the completions of regular Borel measures, defined on a topological $T_1$-space, such that each measure  vanishes on points .

\begin{Cor} Let $Y$ be a 2nd countable topological $T_1$-space and let $E$ be a subset of  $Y$ that is not  universally $\SN$-null. Then it can be written as the union of its two disjoint subsets that cannot be separated by universally $\SN$-measurable sets.

\end{Cor}

\ The CRAS note \cite{Lus34} (S\'eance de 23 avril) was preceded by two other notes by Lusin, \cite{Lus34a} (S\'eance de 19 mars) and \cite{Lus34b} (S\'eance de 26 mars). In those notes  Lusin discussed a new method of solving four `difficiles' (difficult) problems `de la th\'eorie des fonctions' (of the theory of functions). In \cite{Lus34b}, he thought he had solved Problems 1 and 2 on the list, by showing that \textit{every uncountable subset of reals can be decomposed into two disjoint subsets that cannot be separated by Borel sets}\footnote{ It is now known that the result is impossible in ZFC. Indeed, it is consistent \cite{Mar70} that there exists a so called Q-set $E$ which is uncountable, and whose every subset is a relative $G_\delta$ set. In particular, every subset of $E$ is the trace of a Borel subset of real line.}. He mentioned in  a footnote that  Pyotr S. Novikov   also believed having a solution, and indicated that he did not know Novikov's proof.  In the introductory Section 1 of the note \cite{Lus34} that we study in this paper,  Lusin admitted an error in both proofs -- his and Novikov's. He wrote (that in this situation) ``il me semble utile de d\'evelopper compl\`etement la d\'emonstration dans ces cas importants o\`u'' (it seems useful to me  to give a full proof in those important cases when)\textit{ the set is not always 1st category or is not of measure zero}. The next sentence  is: ``Je suivrai de pr\`es la mani\`ere de Monsieur P. Novikoff." In English: ``I will follow the technique of Mr. P. Novikov closely." Thus, it appears that Lusin was able to present proofs of Theorem~\ref{th:LusinDecomposition}
 and  Corollary~\ref{cor:LNT} using the technique that he learned in the meantime from Novikov. For this reason, we
attribute Theorem~\ref{th:LusinDecomposition} together with
Corollary~\ref{cor:LNT} to Lusin \textit{and} Novikov.
And the dual statement of Theorem 3.1 combined with Theorem 5.6  will be referred to as  \textit{the Lusin-Novikov Theorem for Partitions (LNTP)}.

\section{Prikry's manuscript}

As phrased in the Introduction, the statement of Proposition~\ref{th:partition} simplifies the discussion of the early research  triggered by Kuratowski's paper. However, the original formulation in \cite[Theorem 2]{Kur76} is different. That theorem is stated for a partition of an arbitrary 2nd category \textit{subset} of a  Polish space.   As  Theorem~\ref{prop:Lusinlemma2} is  also stated in that generality (Polish space being replaced by a 2nd countable topological   space), it looked, at least on the face of it,  like LNTP was the first  result which openly subsumed Kuratowski's theorem.
Then, we focused on  Prikry's manuscript again. The proof there, although originally applied in the case of $X=[0,1]$, is exceptional for its generality.
Furthermore,  there is a curious passage before Prikry embarks on his statement and proof. He writes that the result is valid, in a stronger form, for any subset $E\subset [0,1]$ whose outer measure is positive or if $E$ is of 2nd category. Namely,  what he claims, amounts in the case of partitions to the non-existence  of the separation by measurable sets or by sets having the Baire property in the sense of the LNTP. However, as he writes it, he restraints himself from going into this more general situation ``in order to avoid extra explanation''.

In conclusion,  Prikry actually proved (compare Corollaries 4.1 and 4.2 in \cite{Kou86})

\begin{Prop}\label{prop:Prikry1}
 Let $Y$ be a 2nd countable topological space equipped with a  finite regular Borel measure $\mu$, and let $\{A_\ga :\ga\in\Gamma\}$  be a point finite family of subsets of   $\,Y$ with union $E=\bigcup_{\ga\in\Gamma}A_\ga$. Suppose that for any $\ga\in\Gamma$, $\mu(A_\ga)=0$ (resp., $A_\ga$ is 1st category) and $\mu_e(E)>0$ (resp., $E$ is 2nd category). Then there exists $\Theta\subset\Gamma$ such that $\,\bigcup_{\ga\in\Theta}A_\ga$ is not measurable (resp., does not have the Baire property).
\end{Prop}
We recall  that $\{A_\ga: \ga\in\Gamma\}$ is said to be \textit{point finite}, if  each point $x$ in $E=\bigcup_{\ga\in\Gamma} A_\ga$ belongs to at most finite number of $A_\ga$'s. Prikry's proof of Proposition \ref{prop:Prikry1} remained in the folklore; it can now be found in \cite[Proof of Lemma 3.2]{Lab17}. Note that although Lemma 3.2 in \cite{Lab17} is stated in a weaker form, its proof gives Propositon~\ref{prop:Prikry1}. An additional hint is that it is better to carry on the proof  with respect to the whole space $Y$ instead of simplifying the notation by taking $\beta=\Gamma$.

\medskip
With the  assumptions of the above Proposition~\ref{prop:Prikry1}, here is the stronger statement claimed by Prikry without proof.

\begin{Th}\label{th:Prikry2} There exists $\Theta\subset\Gamma$ such that there is no $m$-measurable (resp., having the Baire property) subset $C$ of $\,Y$ such that   $C\cap E=\bigcup_{\ga\in\Theta}A_\ga$.
\end{Th}

We will prove it be reducing Theorem~\ref{th:Prikry2}.
 to Propositon~\ref{prop:Prikry1}.

\textit{1. The measure case}.

For the purpose of the proof, it will be convenient to denote the space in the theorem by $X$ (keeping its measure to be $m)$. With this setting,  denote by $\SB$  the $\si$-algebra of Borel subsets of $X$, by $\SB_E$ its restriction to $E$ (i.e., the Borel $\si$-algebra of the subspace $E$ of $X$), and let $\mu$ be defined as the restriction of $m_e$, the outer measure of $m$, to $\SB_E$, that is, $\mu(A)= m_e(A)$ for $A\in \SB_E$.
 Then $\mu$ is a regular Borel measure on $E$ and $\mu(A_\ga)=0$ for each $\ga\in\Gamma$.
Further, define $\SN^\mu=\{Z\subset E: \mu_e(Z)=0\}$ and $\SN^m=\{Z\subset X: m_e(Z)=0\}$.

   By applying  Proposition~\ref{prop:Prikry1} with $Y=E$  and the measure $\mu$, we infer the existence of $\Theta\subset \Gamma$ such that  $\bigcup_{\ga\in\Theta} A_\ga$ is not $\mu$-measurable, i.e., is not a member of the $\mu$-completion of $\SB_E$. Thus, we have $\bigcup_{\ga\in\Theta} A_\ga\not\in \SB_E\vartriangle\SN^\mu$, where $\vartriangle$ stands for the symmetric difference.
Now, denote by $\SM$ the family of $m$-measurable sets in $X$ and observe that
$$
\SM\cap E=(\SB\vartriangle \SN^m)\cap E=\SB_E\vartriangle\SN^\mu.
$$

Hence, there is no $m$-measurable subset $C$ of X such that
 $C\cap E =\bigcup_{\ga\in\Theta}A_\ga$. This ends the proof in the `measure case'.

\medskip

2. \textit{The Baire case}.

 As above, we denote the space in the statement of Theorem \ref{th:Prikry2} by $X$. Let $Q$   be a countable dense subset of $X$ and set $E_1=E\cup Q$. As $E_1$ is dense in $X$, by \cite[\S10.\,IV, Theorem 2]{Kur66}, the sets $A_\ga$ are 1st category relative to $E_1$. By Proposition~\ref{prop:Prikry1} with $Y=E_1$, we get the existence of $\Theta\subset\Gamma$ such that $A=\bigcup_{\ga\in\Theta} A_\ga$ does not have the Baire property relative to $E_1$.
 Suppose that there exists $C$ with the Baire property relative to $X$ such that $C\cap E =A$. Then  $C\cap E_1=A\cup Q_2$, where $Q_2\subset Q$. As $E_1$ is dense in $X$, by \cite[\S11.V, Theorem 2]{Kur66}, $C\cap E_1$ has the Baire property relative to $E_1$, and so $A\cup Q_{2}$ has it. It follows that $A$  has the Baire property relative to $E_1$, in contradiction with the choice of $A$. This ends the próoof in the `Baire case'.

\begin{Rem} One can easily check that, by essentially repeating their proofs, also Theorems~\ref{th:BR} and \ref{th:universal} admit the following strengthening to the point-finite case.
\medskip

\textit{Let $Y$ be a 2nd countable topological (resp.,\,Polish) space. Let $\{A_\ga:\ga\in\Gamma\}$ be a point-finite family of  universally $\SM$-null (resp.,\,always 1st category) subsets of $Y$. If  the union $E=\bigcup_{\ga\in\Gamma}A_\ga$
is  not an $\SM$-null (resp.,\,always 1st category)  subset of $Y$, then there exists $\Theta\subset\Gamma$ such that there is no universally $\SM$-measurable (resp.,\,having  the Baire property in the restricted sense) subset $C$ such that   $C\cap E=\bigcup_{\ga\in\Theta}A_\ga$.}

\end{Rem}

\medskip

\section{Appendix}

\S1. Kunugi's statement in his proof (Section 4), in which he refers to a \textit{result} of Montgomery, is slightly misleading. What he really means is that by applying a similar \textit{method of proof} as the one used by Montgomery in \cite{Mon35}, one can get

\begin{Prop} Let $Y$ be a metric space and $\ep>0$.  There exist  sets $Y^n(\ep), n\in\BN$, such that $Y=\bigcup_n Y^n(\ep)$ and $Y^n(\ep)$ decomposes into a transfinite sequence $ Y^n_\xi(\ep)$ of $G_\rho$-sets in such a way that $d(Y^n_\xi(\ep),Y^n_{\xi'}(\ep))>\frac1n$ for each $\xi\neq\xi'$ and $d(Y^n_\xi(\ep)<\ep$ for each $\xi$.
\end{Prop}

\begin{proof} Let $\{B_\ga:\ga\in\Gamma\}$ be a base of open sets in $Y$ consisting of sets whose diameter $d(B_\ga)<\ep$. Define $H_\ga=B_\ga\sminus \bigcup_{\eta<\ga}B_\eta$ and set
$$
Y^n_\ga(\ep)=\{x\in H_\ga: d(x,Y\sminus B_\ga)>\frac1n\}=H_\ga\cap\{x\in B_\ga: d(x, Y\sminus B_\ga\}>\frac1n\}.
$$
As the latter set is an intersection of a $G_\rho$-set and an open set, $Y^n_\ga(\ep)$ is a $G_\rho$-set.

Some of these sets $Y^n_\ga(\ep)$ may be empty. Removing all the empty sets, we will get a new index set. Let us denote its indices by $\xi$. It is easily seen that, for each $\xi\neq\xi'$, $d(Y^n_\xi(\ep), Y^n_{\xi'}(\ep))>\frac1n$. As $Y^n_\xi(\ep)\subset H_\xi\subset B_\xi$, its diameter is less than $\ep$. If $x\in Y$, then $x\in H_\xi$ for some $\xi$. As $x\in B_\xi$, an open set,  $d(x,Y\sminus B_\xi)>0$. Choose $n_0$ so that $\frac1{n_0}<d(x,Y\sminus B_\xi)$. Then $x\in Y^{n_0}_\xi(\ep)\subset Y^{n_0}(\ep)$. Hence $Y=\bigcup_n Y^n(\ep)$.
\end{proof}

\S2. After the  proof of the existence of two subsets of $E$ that cannot be separated by measurable sets, Lusin gave the following conclusion without any further explanation:

\medskip

\textit{In the end, we have two disjoint subsets $E_1$ and $E_2 $ of $E$ contained in a perfect set $Q$ of positive measure such that each one of them has outer measure equal to the measure of $Q$.}

\medskip
\begin{Rem}  For a stronger result of this type, in whose proof the above conclusion  is used, see \cite[Th\'eor\`eme II]{Sie34a}.
See also \cite{Grze17}.
\end{Rem}

Let us show that Lusin's statement is indeed true (in our more general setting, we will  have to assume that the measure vanishes on points).
We need the following
\begin{Lem}\label{lem:hull} Let $Q$ be a measurable  set contained in an envelope of $A$. Then $Q$ is an envelope of $Q\cap A$.
\end{Lem}

\begin{proof} Denote by $\tilde A$ an envelope of $A$.
Let $G$ be a measurable set contained in $Q\sminus (Q\cap A)$. Then $G\subset\tilde A\sminus A$ and so, by Lemma~\ref{lem:meshull}, $m(G)=0$. The same Lemma implies that $Q$ is an  envelope of $Q\cap A$.
\end{proof}

Now, let $A$ and $ B$  be two disjoint subsets of $E$ that cannot be separated by a measurable set.   Let $S=\tilde A\cap\tilde B$. Then $S$ is of positive measure, because $\tilde A\cap B$ already is, as otherwise $A$ and $B$ could be separated. 
At this point, we assume additionally that ($Y$ is $T_1$ and)  $m$  \textit{vanishes on points}. As $m(S)>0$, we can find a closed subset of $S$ of positive measure  by regularity of $m$ and remove  points that are not its condensation points (cf.\,\cite[1.7.11]{Eng89}) to get a perfect set $Q$. Then, by Lemma~\ref{lem:hull}, $m(Q)=m_e(Q\cap F_1)=m_e(Q\cap F_2)$.
Set $E_1=Q\cap F_1$ and $E_2=Q\cap F_2$.

\end{document}